\documentclass[12pt]{extarticle}
\pdfoutput = 1

\usepackage{amsmath, amsthm, amssymb, hyperref, color}
\usepackage{graphicx}
\usepackage{tikz}
\usepackage{authblk}
\usepackage{xcolor}
\usepackage{booktabs}
\usepackage{tabularx}
\usepackage{tabu}
\usepackage{mathtools}
\usepackage{microtype}
\usepackage[ruled]{algorithm2e}
\usepackage{pgfplots}
\usepgfplotslibrary{fillbetween}
\numberwithin{equation}{section}

\usepackage{caption}
\oddsidemargin \evensidemargin
\usepackage{etoolbox}
\patchcmd{\thebibliography}
  {\settowidth}
  {\setlength{\itemsep}{2pt plus 0.1pt}\settowidth}
  {}{}
\apptocmd{\thebibliography}
  {\small}
  {}{}

\newtheorem{theorem}{Theorem}
\numberwithin{theorem}{section}
\newtheorem{proposition}[theorem]{Proposition}
\newtheorem{lemma}[theorem]{Lemma}
\newtheorem{corollary}[theorem]{Corollary}

\theoremstyle{definition}
\newtheorem{remark}[theorem]{Remark}

\DeclareMathOperator*{\argmin}{argmin}

\DeclareMathOperator{\diag}{diag}
\DeclareMathOperator{\rank}{rank}

\DeclareMathOperator{\prox}{prox}

\newcommand{\R}{\mathbb{R}}

\newcommand{\abs}[1]{{\left\lvert #1 \right\rvert}}
\newcommand{\matop}[1]{{\mathbf{#1}}}

\definecolor{svd}{HTML}{A02C2C}
\definecolor{als}{HTML}{1F77B4}
\definecolor{stiefel}{HTML}{FF7F0E}
\definecolor{riemann}{HTML}{2CA02C}

\date{}
 
\title{\textbf{Riemannian thresholding methods for row-sparse and low-rank matrix recovery}}

\author{Henrik Eisenmann\thanks{Max Planck Institute for Mathematics in the Sciences, 04103 Leipzig, Germany.\\Email: henrik.eisenmann@mis.mpg.de, uschmajew@mis.mpg.de} \quad Felix Krahmer\thanks{Department of Mathematics, Technische Universit\"at M\"unchen, 85748 Garching/Munich, Germany.\\Email:  felix.krahmer@tum.de} \and  Max Pfeffer\thanks{Faculty of Mathematics, Technische Universit\"at Chemnitz, 09107 Chemnitz, Germany.\\Email: max.pfeffer@math.tu-chemnitz.de} \quad Andr\'e Uschmajew${}^*$}

\begin{document}
\maketitle

\begin{abstract}
\noindent
In this paper, we present modifications of the iterative hard thresholding (IHT) method for recovery of jointly row-sparse and low-rank matrices. In particular a Riemannian version of IHT is considered which significantly reduces computational cost of the gradient projection in the case of rank-one measurement operators, which have concrete applications in blind deconvolution.
Experimental results are reported that show near-optimal recovery for Gaussian and rank-one  measurements, and that adaptive stepsizes give crucial improvement. A Riemannian proximal gradient method is derived for the special case of unknown sparsity. 
\end{abstract}

\noindent

\section{Introduction}
Since the seminal works on compressive sensing by Cand\`es, Romberg, and Tao \cite{candes2006compressed} and by Donoho \cite{donoho2006compressed}, the question of recovering structured signals from subsampled random measurements has received significant attention. Two structural models of fundamental importance in applications are sparse signals and low-rank matrices. A sparse signal is one that can be well approximated by a linear combination of just a few elements in a given basis or dictionary, and has proven to be appropriate for example in magnetic resonance imaging or remote sensing. Low-rank matrix models have been successful, e.g., for recommender systems and in applications related to phase retrieval and wireless communication. In these last two areas, the low-rank model arises from lifting, that is, a quadratic or bilinear measurement is equivalently expressed as a linear function acting on the rank-one matrix formed from the outer product of the two inputs. Consequently, combined with a sparsity assumption for the underlying signal (or signals), this entails that the matrix to be recovered is simultaneously of low rank and row and/or column sparse.

In this paper, we focus on the low rank and row sparse scenario. Such a model arises for example in wireless communication as follows. When an encoded message is transmitted via an unknown channel, the received signal can be modeled as the convolution of the encoded message vector with a channel vector. For this vector, sparsity can be assumed when only few transmission paths are active. The goal is then to estimate both the message and the sparse channel vector from the received signal. This problem of blind deconvolution can be  recast into a recovery problem for a row-sparse rank-one matrix from linear measurements (see section~\ref{sec: blind deconvolution}).
For a subspace model instead of a sparsity model (that is, when the active  transmission paths are assumed to be known), a number of recent works have discussed solution strategies, including lifting \cite{Ahmed2014, KS19} and nonconvex methods \cite{LLSW19}. Subsequently, these methods have been generalized to the more difficult case of multiple simultaneous transmissions \cite{ling2017blind, ling2017regularized, Jung2018}, but again only for subspace models.

To make our model precise,  
we consider the space $\R^{M \times N}$ of $M \times N$ matrices and denote by
\(
\| X \|_{0} 
\)
the number of nonzero rows of $X$. If $\| X \|_0 \le s$, we say that $X$ is \emph{row $s$-sparse}. The set of row $s$-sparse matrices is denoted by
\[
\mathcal N_s = \{ X \in \R^{M \times N} \colon \| X \|_0 \le s \}.
\]
The set of matrices of rank at most $k$ is denoted by
\[
\mathcal M_k = \{ X \in \R^{M \times N} \colon \rank(X) \le k \}.
\]
In this paper we focus on the intersection of these sets,
\[
\mathcal M_{k,s} = \mathcal M_k \cap \mathcal N_s = \{ X \in \R^{M \times N} \colon \rank(X) \le k, \ \| X \|_0 \le s \}.
\]
Throughout, we  assume that
\[
k<s,
\]
since otherwise the low-rank constraint is void. The problem we then consider is to recover a given matrix $X \in \mathcal M_{k,s}$ from $m$ linear measurements
\[
\langle A_p, X \rangle_F = y_p, \quad p = 1,\dots,m,
\]
where $A_1,\dots,A_m \in \mathbb R^{M \times N}$, and $\langle \cdot, \cdot \rangle_F$ is the usual Frobenius inner product. With a corresponding linear operator $\matop{A} : \R^{M \times N} \rightarrow \R^m$ this can be formulated as solving the problem
\begin{equation}\label{eq: problem}
	\matop A (X) = y, \quad X \in \mathcal M_{k,s}.
\end{equation}
for a given $y \in \R^m$.

The two simultaneous structural constraints defining $\mathcal M_{k,s}$ significantly reduce the degrees of freedom and allow for injectivity of $\matop A$ on $\mathcal M_{k,s}$ given a considerably smaller number of measurements $m$ as compared to constraining only on one of the two sets $\mathcal M_k$ or $\mathcal N_s$. Injectivity properties have been shown to require only $O(k(s+N))$ measurements in various scenarios \cite{EM14, LLB16}; for generic measurement operators, precise conditions on the number of measurements are known~\cite{kech2017optimal}.

At the same time, the simultaneous objectives make it harder to practically recover at near-minimal sampling complexity. In particular, while both the low-rank and the sparsity objective on its own admit tractable convex relaxations with recovery guarantees under random measurements, it has been shown that no linear combination of these two objectives allows for comparable guarantees for the joint objective \cite{Oymak2015}, see also~\cite{Klietsch2019}. Greedy-type methods are also difficult to generalize to the joint minimization problem. A typical key step in these methods is a projection onto the set of admissible signals. For sparsity and low-rank models, this projection can be efficiently implemented by restricting to the largest coefficients or the largest principal components, respectively. For the joint low-rank and (bi-)sparse model, however, this projection is an instance of the \emph{Sparse Principal Component Analysis} problem, which is known to be NP hard in general \cite{Magdon-Ismail2017}. 

For very special measurements, certain two-stage procedures can allow for guaranteed recovery. For phase retrieval, this works when  measurements of the form $|b_i^* \Phi x|^2$, i.e., $A_p= b_p^*\Phi\Phi^*b_p$ and $X=xx^*$ in terms of the representation \eqref{eq: problem}, are considered with $\Phi$ representing a linear dimension reduction, and the number of measurements is larger than the embedding dimension of $\Phi$ by at least a constant factor \cite{IVW17}. Namely, for $\Phi\in \R^{m\times N}, m\gtrsim s\log \tfrac{N}{s}$, 
 and $b_i$ both chosen with i.i.d.~Gaussian entries, such measurements allow the recovery of $y=\Phi x$ via standard phase retrieval techniques, from which one can then infer $x$ via compressive sensing. Similar nested measurements can also be constructed in the framework of bilinear problems~$\cite{BR16}$. While arguably such very special measurements cannot be assumed in many scenarios of interest, these observations show that solving sparse bilinear problems is not an intrinsically hard problem in all cases. 

That said, some recent progress has been made also for more generic classes of measurements. A number of works have established local recovery guarantees for a near-optimal number of measurements, that is, convergence to the true solution is guaranteed from all starting points in a suitable neighborhood.
For sparse phase retrieval, such guarantees have been established for gradient descent \cite{S19} and {\em Hard Thresholding Pursuit} \cite{CLLY20}. For unstructured Gaussian measurements, local guarantees are available for the alternating algorithms {\em Sparse Power Factorization} \cite{Lee2018} and {\em Alternating Tikhonov regularization and Lasso} \cite{FMN21}. Suitable initialization procedures to complement these methods by constructing a starting point in a small enough neighborhood of the solution, however, are known only for certain special classes of signals such as signals with few dominant entries \cite{Lee2018, GKS19}. In~\cite{Maly2021} a model of low-rank recovery with essentially sparse nonorthogonal factors is considered, for which a robust injectivity property for several types of measurements is established. We also mention the work~\cite{Haeffele2020} in which a rank-adaptive algorithm for finding global minima of nonconvex formulations of structured low-rank problems is presented.

Despite the recent progress, it remains largely an open problem whether and how joint (bi-)sparse and low rank signals can be efficiently recovered from a near-minimal number of measurements when no such initialization is provided. For an in-depth discussion of what makes the problem difficult and some initial ideas regarding how to solve it, we refer the reader to~\cite{Foucart2020}.

\subsubsection*{Contribution and outline} 

In this paper we consider a class of non-convex iterative methods based on modification of \emph{Iterative Hard Thresholding} (IHT) as proposed in the recent work~\cite{Foucart2020}. In principle, under suitable RIP assumptions for the operator $\matop A$, the standard IHT method could be used to approximate the solution of~\eqref{eq: problem} at an exponential rate. The main obstacle is that the exact projections on the set $\mathcal M_{k,s}$ are NP hard to compute as mentioned above. It is, however, possible to compute quasi-optimal projections on $\mathcal M_{k,s}$ by simply using the successive projections on $\mathcal M_k$ and $\mathcal N_s$, or vice versa. This approach is taken in section~\ref{sec: IHT approach} where we first derive quasi-optimality constants for such projections. These results complement some of the investigations in~\cite{Foucart2020} on the bisparse case.  We then consider a practical version of IHT that uses these quasi-optimal projections in combination with line search, and present a local convergence result for such a method.

The main contribution of this paper is a further  modification of the IHT algorithm that makes use of the manifold properties (of the smooth part) of the set $\mathcal M_k$ by applying a tangent space projection to the search direction. This idea is inspired by Riemannian low-rank optimization, which has been shown to be efficient in several applications, including matrix completion and matrix equations; see~\cite{bookch2020} for an overview. We demonstrate that in the important case of rank-one measurements, which includes problems of blind deconvolution, the additional tangent space projection allows for a significant reduction of computational cost since the projection of the gradient onto the tangent space can be efficiently realized even for large low-rank matrices. This observation does not specifically rely on the sparsity structure and should therefore be of interest for other low-rank  recovery problems with rank-one measurements as well. The proposed Riemannian version of IHT is presented in section~\ref{sec: Riemannian IHT}, with a detailed discussion for the case of rank-one measurements in section~\ref{sec: complexity}.

Lastly, we also consider the scenario that the sparsity parameter $s$ is unknown. One can then replace the hard-thresholding operator for the rows with a soft-thresholding operator. As we show in section~\ref{sec: Riemannian prox-gradient method}, such a modification admits a natural interpretation as a manifold proximal gradient method on $\mathcal M_k$ with the $(1,2)$-norm as a penalty. Notably, in contrast to other recent generalizations of the proximal gradient method to manifolds \cite{Chen2020,Huang2020}, the structural constraints considered in this work allow for a closed-form expression of the proximal step via soft-thresholding.

Finally, in section~\ref{sec: numerical experiments} we present several numerical experiments. We test the three algorithms proposed in this work in scenarios with random measurements as well as rank-one measurements. This also includes a numerical experiment on blind deconvolution with Fourier measurements.

The main outcome of our results is that in practice, and in the noiseless case, the proposed variants of IHT are capable of recovering row-sparse low-rank matrices with a near optimal number of measurements, up to a constant oversampling factor. The theoretical guarantees are currently restricted to local convergence results, and will be subject to future research.

\section{Review of iterative hard thresholding approaches}\label{sec: IHT approach}

The sparse low-rank recovery problem \eqref{eq: problem} can be recast into the optimization problem
\begin{equation}\label{eq: cost function IHT}
\min f(X) = \frac{1}{2} \| \matop A(X) - y \|_2^2 \quad \text{s.t. $X \in \mathcal M_{k,s}$,}
\end{equation}
where $\|\cdot\|_2$ denotes the Euclidean norm in $\R^m$. Noting that 
\[
\nabla f(X)= \matop A^*(\matop A (X) - y),
\]
an intuitive approach to the sparse low-rank recovery is the iterative hard thresholding method, which takes the form
\begin{equation}\label{eq: IHT general}
	X_{\ell+1} = \matop P_{\mathcal M_{k,s}}(X_\ell - \matop A^*(\matop A (X_\ell) - y)).
\end{equation}
where $\matop P_{\mathcal M_{k,s}}$ is a metric projection on $\mathcal M_{k,s}$, which is characterized by the best approximation property
\[
\| X - \matop P_{\mathcal M_{k,s}}(X) \|_F \le \| X - Y \|_F \quad \text{for all $Y \in \mathcal M_{k,s}$.} 
\]
By the standard arguments, one can show that under a suitable RIP assumption this method is globally convergent to the solution $X^*$.

 A main obstacle is that the projection on the set $\mathcal M_{k,s}$ is usually prohibitively expensive to compute, essentially at the cost of checking almost all possible subsets of $s$ rows of $X$. For a subset $S \subseteq \{ 1,\dots,M\}$ we denote by $X_S$ the projection of $X$ where all rows not in $S$ have been set to zero. We then have the following result, which has already been noted in~\cite{Foucart2020} for the case $k=1$.

\begin{proposition}
For given $X$, $\matop P_{\mathcal M_{k,s}}(X)$ is given as a best rank $k$ approximation of $X_{S'}$, where the submatrix $X_{S'}$ maximizes $\sigma_1^2(X_S) + \dots + \sigma_k^2(X_S)$ (sum of squares of largest singular values) among all submatrices $X_S$ of $X$ with $\abs{S} = s$.
\end{proposition}

\begin{proof}
For any row support set $S$ with $\abs{S} \le s$, the optimal closest point in $\mathcal M_{k,s}$ with this support is obviously a best rank $k$ approximation $\matop T_k(X_S)$ of $X_S$ (which has the same row support). It has the squared distance
\begin{align*}
\| X - \matop T_k(X_S) \|_F^2 &= \| X_{\bar S} \|_F^2 + \| X_S - \matop T_k(X_S) \|_F^2 \\ &= \| X_{\bar S} \|_F^2 + \| X_S \|_F^2 -\| \matop T_k(X_S) \|_F^2 = \| X \|_F^2 - \| \matop T_k(X_S) \|_F^2,
\end{align*}
where $\bar S$ is the complement of $S$. This shows that the minimum is achieved when $ \| \matop T_k(X_S) \|_F^2=\sigma_1^2(X_S) + \dots + \sigma_k^2(X_S)$ is maximal among all $S$ with $\abs{S} \le s$. However, since this quantity does not decrease when adding rows to a matrix, it suffices to take the maximum over $\abs{S} = s$.
\end{proof}

By the above proposition, the projection $\matop P_{\mathcal M_{k,s}}$ is in principle available by computing the $k$ largest singular vectors of all possible submatrices with $s$ rows, which has combinatorial complexity. Even if a smaller set of candidates for the rows, say $2s$ of them, could be identified beforehand, the complexity remains exponential in $s$, not even counting the cost for computing the singular vectors. The computation of $\matop P_{\mathcal M_{k,s}}$ should therefore be in general infeasible, which also makes it infeasible to compute \eqref{eq: IHT general}.

\subsection{Quasi-optimal projections}

Feasible variants of IHT can be obtained by employing projections on~$\mathcal M_{k,s}$ that are only quasi-optimal, an  idea already suggested in~\cite{Foucart2020}. Such variants are derived from the fact that $\mathcal M_{k,s}$ is the intersection of the two cones $\mathcal N_s$ (row $s$-sparse matrices) and $\mathcal M_k$ (rank-$k$ matrices), and for both sets the metric projections are explicitly available. For $\mathcal N_s$ it is given as
\[
\matop H_s(X) = X_S
\]
where $S \subseteq \{1,\dots,M\}$ contains indices of $s$ rows of $X$ with largest norm. For~$\mathcal M_k$ the best rank-$k$ approximation of $X$ can be computed from the dominant singular vectors as usual and is denoted by $\matop T_k(X)$. Both $\matop H_s$ and $\matop T_k$ are nonlinear maps.  They are  possibly set-valued, in which case we assume  that some specific selection rule is applied. Since computing a best rank-$k$ approximation does not increase the row support of a matrix, the composition
\[
\matop P_{k,s}(X) := (\matop T_k \circ \matop H_s)(X)
\]
always maps into the cone $\mathcal M_{k,s}$. Similarly, projecting onto the largest $s$ rows does not increase the rank, hence the map
\[
{\hat{\matop P}}_{k,s}(X) := (\matop H_s  \circ \matop T_k)(X)
\]
also maps into $\mathcal M_{k,s}$. 

Computationally, $\matop P_{k,s}(X)$ is obtained from $X$ by first restricting to the submatrix consisting to the  $s$ rows of largest norm, and then computing a best rank-$k$ approximation of that submatrix. Since this submatrix has only $s$ rows this reduces the cost of the SVD. In contrast, ${\hat{\matop P}}_{k,s}(X)$ requires first a truncated SVD $U\Sigma V^T$ of $X$, which is in general more expensive. However, there is a potential scenario when ${\hat{\matop P}}_{k,s}(X)$ is applied to tangent vectors of the fixed rank-$k$ manifold, where this step is cheap. Also note that for finding the largest $s$ rows it is then sufficient to determine the largest rows of the matrix $U \Sigma$, which has only $k$ columns, so this step becomes slightly cheaper too.

The following proposition shows that both $\matop P_{k,s}$ and $\hat{\matop P}_{k,s}$ are quasi-optimal projections. This has already been shown in~\cite[Prop.~12]{Foucart2020}. We include a proof below, since the setting with only row sparsity that we consider in this paper allows to exploit a certain  commutativity relation that is not available in the general bisparse case and leads to an improved quasi-optimality constant compared to the result in~\cite{Foucart2020}.

\begin{proposition}\label{prop: quasioptimal}
For any $X \in \R^{M \times N}$ the projections $\matop P_{k,s}$ and $ \hat{\matop P}_{k,s}$ map into $  \mathcal M_{k,s}$ and are quasi-optimal in the sense that
\[
\| X - \matop P_{k,s}(X) \|_F \le \sqrt{2} \| X - \matop P_{\mathcal M_{k,s}}(X) \|_F, \qquad \| X - {\hat{\matop P}}_{k,s}(X) \|_F \le \sqrt{2} \| X - \matop P_{\mathcal M_{k,s}}(X) \|_F. 
\]
\end{proposition}

\begin{proof}
We first observe that both nonlinear mappings $\matop H_s$ and $\matop T_k$ for every input $X$ in fact act as linear orthogonal projections in the space $\R^{M \times N}$. Indeed, for given $X$, we can write
\[
\matop H_s(X) = D_{X} X, \quad \text{and} \quad \matop T_k(X) = X V_X^{} V_X^\top, 
\]
where $D_{X}$ is a binary diagonal matrix that selects $s$ rows supporting the $s$ largest (in norm) rows of $X$, and $V_X$ consists of the leading $k$ right singular vectors of $X$. In the rest of the proof we write $D$ instead of $D_{X}$ and $V$ instead of $V_X$. We first consider the map
\(
X \mapsto D X V V^\top
\)
and show that it provides an (alternative) quasi-optimal projection. Note that
\begin{align*}
\| X - D X VV^\top \|_F^2  &= \| X(I - VV^\top)\|_F^2 + \| (I - D)X V V^\top \|_F^2 \\ &\le \| X(I - VV^\top)\|_F^2 + \| (I - D)X \|_F^2.
\end{align*}
Since $\mathcal M_{k,s} \subseteq \mathcal M_k$, we have
\begin{equation*}\label{eq: first est}
\| X - X VV^\top \|_F= \min_{Y \in \mathcal M_k} \| X - Y \|_F \le \min_{Y \in \mathcal M_{k,s}} \| X - Y \|_F \le \| X - \matop P_{\mathcal M_{k,s}} (X) \|_F.
\end{equation*}
By an analogous argument, since $\mathcal M_{k,s} \subseteq \mathcal N_s$, we also have  $\| X - DX \|_F \le \| X - \matop P_{\mathcal M_{k,s}} (X) \|_F$. Therefore we obtain
\[
\| X - D X VV^\top \|_F \le \sqrt{2} \| X- \matop P_{\mathcal M_{k,s}} (X) \|_F.
\]
To conclude the proof, it remains to show that
\[
\| X - \matop P_{k,s}(X) \|_F \le \| X - D X VV^\top \|_F\quad\text{and} \quad\| X - \hat{\matop P}_{k,s}(X) \|_F \le \| X - D X VV^\top \|_F.
\]

Since $\matop P_{k,s}(X) = \matop T_k(DX)$ is supported in the same rows as $DX$, we have the orthogonal decomposition
\begin{align*}
\| X - \matop P_{k,s}(X) \|_F^2 &= \| X - DX  \|_F^2 + \|DX - \matop T_k(DX) \|_F^2.
\end{align*}
The second term on the right can be estimated as
\[
\|DX -  \matop T_k(DX) \|_F^2 \le \|DX - DX VV^\top \|_F^2
\]
since $DX VV^\top$ is a rank-$k$ matrix.  
It thus follows that
\[
\| X - \matop P_{k,s}(X) \|_F^2 \le \| X - DX  \|_F^2 + \|DX - DX VV^\top \|_F^2 = \| X - DX VV^\top \|_F^2.
\]
Similarily, since $\hat{\matop P}_{k,s} (X)= \matop H_s(X VV^\top)$, we have that
\[
\| X - \hat{\matop P}_{k,s}(X) \|_F^2 = \| X - XVV^\top  \|_F^2 + \|XVV^\top - \matop H_s(X VV^\top) \|_F^2.
\]
The second term on the right is not larger than $\|XVV^\top - D X VV^\top \|_F^2$, which likewise shows
\[
\| X -\hat{\matop P}_{k,s}(X) \|_F \le \| X - D X VV^\top \|_F,
\]
as desired.
\end{proof}

\begin{remark}
It is is interesting to note that for $\matop P_{k,s}$ the constant $\sqrt{2}$ is not attained for matrices $X$ where $\matop H_s(X)$ is single valued. If it were attained, then the proof shows that we have $\| X - DX \|_F = \| X - \matop P_{\mathcal M_{k,s}} (X) \|_F$, that is $\min_{Y \in \mathcal N_{s}} \| X - Y \| _F= \min_{Y \in \mathcal M_{k,s}} \| X - Y \|_F$. Then, however, $\matop P_{k,s}(X) = \matop H_s(X)$ is an optimal projection. Similarly, the constant $\sqrt{2}$ is not attained for $\hat{\matop P}_{k,s}$ when $\matop T_k$ is single valued.
\end{remark}

\subsection{IHT with adaptive stepsize}

Using the quasi-optimal projector $\matop P_{k,s}$, one obtains a modified version of IHT shown in Algorithm~\ref{alg: IHT quasi}, in which we additionally include a step size control. In principle, one could use the projector $\hat{\matop P}_{k,s}$ instead, but as noted above it should usually be more expensive to compute unless further structure can be exploited. In principle any starting point $X_0 \in \mathcal M_{k,s}$ could be used, but we noted that in our experiments the proposed choice $X_0 = 0$ works well.

\begin{algorithm}
\SetKwInOut{Input}{Input}
\Input{Linear operator $\matop{A}$, measurements $y$,\\ starting point $X_0 = 0 \in \mathcal{M}_{k,s}$ }
\For{$\ell = 0,1,\dots$}
{
Compute
\[
X_{\ell + 1} = (\matop T_k \circ \matop H_s)(X_\ell - \alpha_\ell \matop{A}^*(\matop{A}(X_\ell) - y))).
\]
}
\caption{IHT with quasi-optimal projection}\label{alg: IHT quasi}
\end{algorithm}

Possible step sizes in Algorithm~\ref{alg: IHT quasi} are either
$\alpha_\ell = 1$
(as in classical IHT) or $
\alpha_\ell = \frac{ \| \matop{A}^*(\matop{A}(X_\ell) - y)) \|^2}{\| \matop{A}( \matop{A}^*(\matop{A}(X_\ell) - y))) \|^2 },
$
which yields the optimal step size without projection. In our experiments, this did however not significantly improve the success or speed of convergence. Instead, we found that an adaptive line search works well. We implemented an Armijo backtracking where $\alpha_\ell = \beta^p$ and  $p$ is the smallest nonnegative integer that fulfills 
\[
f(X_\ell) - f\bigl((\matop T_k \circ \matop H_s)(X_\ell - \beta^p  \matop{A}^*(\matop{A}(X_\ell) - y)))\bigr) \geq \gamma \, \beta^p \, \| \matop{A}^*(\matop{A}(X_\ell) - y)\|^2
\]
for parameters $\beta \in (0,1), \gamma > 0$. We choose $\beta = 0.5$ and $\gamma = 10^{-4}$ in our experiments. Note that the projection is included in the Armijo condition, but the search direction is not guaranteed to be  a descent direction for $f\circ\matop T_k \circ \matop H_s$. In cases where the Armijo condition cannot be fulfilled, we resort to the regular stepsize rule $\alpha_\ell = 1$.

The most costly steps in Algorithm~\ref{alg: IHT quasi} are the computation $\matop A^*(\matop A(X)-y)$ and the quasi-optimal projections. We consider this in more detail in section~\ref{sec: complexity}.

\subsection{Convergence}

To our knowledge, for general measurements no global convergence result is currently available for Algorithm~\ref{alg: IHT quasi}, nor for any other algorithm in a near-minimal parameter regime. However, in the noiseless case, and with constant step size $\alpha_\ell = 1$, it is easy to state a qualitative \emph{local} convergence result under a RIP assumption. We say that $\matop A$ satisfies a $\delta_{k,s}$-RIP on $\mathcal{M}_{k,s}$ if 
\[
(1 -\delta_{k,s}) \|X\|_F^2 \le \| \matop{A} (X)\|^2 \le (1+\delta_{k,s}) \|X\|_F^2 \quad \text{for all $X \in \mathcal M_{k,s}$.}
\]
One can show that Gaussian measurements will satisfy a $\delta_{k,s}$-RIP with high probability if the number of measurements is at least of order $\delta^{-2}_{k,s}k(s + N) \ln(MN)$, cf.~\cite[Theorem 2]{Lee2018}. 

The local convergence proof is based on the simple observation that in a sufficiently small neighbourhood of a matrix with $s$ nonzero rows the quasi optimal projection $\matop P_{k,s}= \matop T_k \circ \matop H_s$ is indeed the optimal projection on $\mathcal M_{k,s}$, since the correct rows are selected.

\begin{lemma}\label{lemma: localoptimal proj}
Let $X^*\in \mathcal{M}_{k,s}$ have exactly $s$ nonzero rows  and let $\mu$ be the smallest norm among the nonzero rows of $X^*$. If $Y\in\R^{M\times N}$ satisfies $\|Y-X^*\|_F < \frac{\mu}{2}$, then $(\matop T_k\circ \matop H_s)(Y)=\matop P_{\mathcal{M}_{k,s}}(Y)$.
\end{lemma}
\begin{proof}
Let $S$ be the row support of $X^*$. Obviously the largest $s$ rows of $Y$ are supported in $S$ and $(\matop T_k\circ \matop H_s)(Y)$ provides the best approximation with respect to this row support. We therefore need to show that the best approximation $\matop P_{\mathcal{M}_{k,s}}(Y)$ also has this row support. Indeed, let $Z$ be any matrix with a different support of size at most $s$ and let $y$ be any row of $Y$ not in the row support of $Z$ (but supported in $S$). Then $\|Y-Z\|_F\geq \|y\|\geq \frac{\mu}{2} >  \|Y-X^*\|_F \geq (\matop T_k \circ \matop H_s) (Y)$. This implies that $\matop P_{\mathcal{M}_{k,s}}(Y)$ needs to be supported in $S$.
\end{proof}

\begin{corollary}\label{cor: local convergence IHT}
Let $\delta_{3k,3s}<0.5$ be the RIP constant of $\matop{A}$ for the set $\mathcal{M}_{3k,3s}$. Let $X^*$ be the (unique) solution of~\eqref{eq: problem} with exactly $s$ nonzero rows, and assume $\|X_0-X^*\|_F\leq \frac{\mu}{2\|\matop{I}-\matop{A}^*\matop{A}\|_{}}$, where $\mu$ is the smallest norm among the nonzero rows of $X^*$. Then the sequence generated by Algorithm~\ref{alg: IHT quasi} with a fixed step-size $\alpha_\ell = 1$ 
satisfies
\[
\|X_{\ell+1}-X^*\|_F \leq 2\delta_{3k,3s} \|X_{\ell}-X^*\|_F.
\]
\end{corollary}
 \begin{proof}
The proof is adapted from~\cite[Thm.~6.15]{FoucartBook}. Let $V$ be a linear subset of $\mathcal M_{k,s} + \mathcal M_{k,s} +\mathcal  M_{k,s}\subset \mathcal M_{3k,3s}$.
The RIP implies the spectral bounds
\begin{equation*}
-\delta_{3k,3s} \|X\|_F^2\leq\langle X , \matop A^*(\matop A(X))-X\rangle_F=\langle X ,(\matop A_{V}^*\matop A_{V}-\matop I_{V})(X)\rangle_F\leq\delta_{3k,3s} \|X\|_F^2
\end{equation*}
for all $X \in V$. Hence, 
the restricted operator $\matop A_{V}$ satisfies the estimate
\begin{equation*}
\|(\matop A_{V})^*\matop A_{V}-\matop I_{V}\|_{V\to V}\leq \delta_{3k,3s}
\end{equation*}
for the operator norm. This replaces the use of \cite[Lem.~6.16]{FoucartBook} in the proof of~\cite[Thm.~6.15]{FoucartBook}. 
Lemma~\ref{lemma: localoptimal proj} implies, that for  $\|X_\ell-X^*\|_F\leq \frac{\mu}{2\|\matop{I}-\matop{A}^*\matop{A}\|_{}}$ the quasi-optimal projection is indeed optimal. Hence, the proof technique of~\cite[Thm.~6.15]{FoucartBook} can be applied.
\end{proof}

The asymptotic rate of convergence however is faster than suggested by Corollary~\ref{cor: local convergence IHT}. To see this, let $\mathcal{M}_{k,S^*}$ denote the variety of matrices with rank at most $k$ and a fixed row support $S^*\subseteq \{1,\dots,L\}$ with $\abs{S^*} = s$ such that  $X^*\in\mathcal{M}_{k,S^*}$. If $\rank(X^*) = \min(k,s)$, then $\mathcal{M}_{k,S^*}$ is a smooth manifold around $X^*$ and the 
asymptotic rate depends on a RIP constant of the tangent space of $T_{X^*}\mathcal{M}_{k,S^*}$.
 
\begin{proposition}\label{prop: asymptotic rate iht}
Let $X^*$  be a solution of~\eqref{eq: problem} with exactly $s$ nonzero rows and assume $\rank(X^*) = \min(k,s)$. Assume the spectral norm of $\matop I-\matop{A}^*\matop A$ on $T_{X^*}\mathcal{M}_{k,S^*}$ is $\delta < 1$.  There exists an $\epsilon>0$ such that if the  the sequence $(X_{\ell})$  generated by Algorithm~\ref{alg: IHT quasi} with stepsize $\alpha_\ell=1$ satisfies \(\|X_\ell-X^*\|_F\le \epsilon\) for some $\ell$, then $X_\ell$ converges to $X^*$ and 
$\limsup_{\ell\to\infty}\frac{\|X_{\ell+1}-X^*\|_F}{\|X_{\ell}-X^*\|_F}\leq \delta$.
\end{proposition} 

Note that $T_{X^*}\mathcal{M}_{k,S^*} \subseteq \mathcal{M}_{2k,s}$ and hence $\delta \le \delta_{2k,s}$ as for linear spaces the RIP constant and the spectral norm of $\matop I-\matop{A}^*\matop{A}$ coincide. In fact, for Gaussian measurements, the embedding dimension needed to obtain a spectral norm bounded by $\delta$ with high probability does not require a logarithmic factor, which is why one can generally expect $\delta$ to be smaller than $\delta_{2k,s}$ by a square root $\log$ factor.

 \begin{proof}
 Lemma~\ref{lemma: localoptimal proj} implies that in proximity to the solution $X^*$ the quasi-optimal projection $\matop T_k \circ \matop H_s$ equals the best approximation $\matop P_{\mathcal M_{k,S^*}}$ in Frobenius norm onto the manifold ${\mathcal M_{k,S^*}}$.
For $X_\ell$ close enough to $X^*$ we then get
 \begin{align*}
 X_{\ell+1}-X^*&=(\matop T_k \circ \matop H_s)(X_\ell-\matop A^*\matop A(X_\ell -X^*))-X^*\\
 &=\matop P_{T_{X^*}\mathcal{M}_{k,S^*}}(X_\ell-\matop A^*\matop A(X_\ell -X^*))-X^*+o(\|(\matop I -\matop A^* \matop A)(X_\ell-X^*)\|_F)\\
 &=\matop P_{T_{X^*}\mathcal{M}_{k,S^*}}(X_\ell-\matop A^*\matop A(X_\ell -X^*))-X^*+o(\|X_\ell-X^*\|_F)
 \end{align*}
 by linearizing the projection $\matop P_{\mathcal M_{k,S^*}}$, see~\cite[Lemma 4]{Absil2012}. Next we exploit that $X^* \in T_{X^*}\mathcal{M}_{k,S^*}$ and $(\matop I-\matop P_{T_{X^*}\mathcal{M}_{k,S^*}})(X_\ell-X^*) = o(\|X_\ell-X^*\|_F)$ (see, e.g.~\cite[Lemma 4.1]{Wei2016}) to get
 \begin{align*}
 X_{\ell+1}-X^*&=\matop P_{T_{X^*}\mathcal{M}_{k,S^*}}(\matop I-\matop A^*\matop A )\matop P_{T_{X^*}\mathcal{M}_{k,S^*}}(X_\ell -X^*)+o(\|X_\ell-X^*\|_F).
 \end{align*}
 By assumption,  
$\|\matop P_{T_{X^*}\mathcal{M}_{k,S^*}}(\matop I-\matop A^*\matop A )\matop P_{T_{X^*}\mathcal{M}_{k,S^*}}\|_F=\delta<1$. This allows to prove to assertion by induction.
 \end{proof} 
 
\section{Riemannian optimization approaches}\label{sec: Riemannian approach}

It is possible to exploit the structure of the set $\mathcal M_{k,s}$ as an intersection of $\mathcal{M}_k$ and $\mathcal{N}_s$. Since the smooth part of the set $\mathcal{M}_k$ (matrices of rank equal to $k$) is a connected manifold, it is reasonable to replace the negative gradient $- \nabla f(X_\ell) = -\matop{A}^*(\matop{A}(X_\ell) - y)$ by a Riemannian gradient, 
 that is, by its projection on a tangent space. When using the Riemannian metric inherited from the embedding into Euclidean space, the Riemannian gradient is simply given as the orthogonal projection of the Euclidean gradient onto the tangent space of $\mathcal M_k$ at $X_\ell$. In this way we obtain modifications of IHT with tangential search directions.

\subsection{Riemannian IHT}\label{sec: Riemannian IHT}

Given the SVD $X_\ell=U^{}_\ell\Sigma^{}_\ell V_\ell^\top$, and assuming $\rank (X_\ell) = k$, the orthogonal projection onto the tangent space is the linear map
\begin{equation}\label{eq: tangent space projector}
\matop{P}_\ell (Z) = U^{}_\ell U_\ell^\top Z + ZV^{}_\ell V_\ell^\top  -U^{}_\ell U_\ell^\top ZV^{}_\ell V_\ell^\top,
\end{equation}
see, e.g.,~\cite{Vandereycken2013}. Note that if $k$ is small, then the computation of the projection requires multiplication by tall matrices only. For the cost function~\eqref{eq: cost function IHT}, the projected gradient is
\[
\matop P_\ell(\nabla f(X_\ell)) = \matop P_\ell (\matop{A}^*(\matop{A}(X_\ell) - y)).
\]
Replacing the gradient in IHT with this projected gradient results in the scheme displayed in Algorithm~\ref{alg: semi-riem IHT}.

\begin{algorithm}
\SetKwInOut{Input}{Input}
\Input{Linear operator $\matop{A}$, measurements $y$,\\ starting point $X_1= (\matop T_k \circ \matop H_s) (\alpha_0 \matop{A}^* y) \in \mathcal{M}_{k,s}$ with initial step size $\alpha_0 \in \R$ }
\For{$\ell = 1,2,\dots$}
{
Choose stepsize $\alpha_\ell > 0$\;
Compute
\[
X_{\ell + 1} = (\matop T_k \circ \matop H_s) (X_\ell - \alpha_\ell \matop{P}_\ell (\matop{A}^*(\matop{A}(X_\ell) - y))).
\]
}
\caption{Riemannian IHT with quasi-optimal projection}\label{alg: semi-riem IHT}
\end{algorithm}

\noindent
Possible step size rules are again constant steps $\alpha_\ell=1$ or an Armijo-like condition
\begin{equation}\label{eq:riemarm}
f(X_\ell) - f\bigl((\matop T_k \circ \matop H_s)(X_\ell - \beta^p \matop P_\ell( \matop{A}^*(\matop{A}(X_\ell) - y)))\bigr) \geq \gamma \, \beta^p \, \| \matop P_\ell( \matop{A}^*(\matop{A}(X_\ell) - y))\|_F^2.
\end{equation}
Without further structure, the tangent space projection has a cost of $O((M+N)k^2)$ flops.  An advantage of this approach is that application of the quasi-optimal projection $\matop T_k \circ \matop H_s$ then becomes somewhat cheaper. Indeed, since $\matop P_\ell(X_\ell) = X_\ell$, and since elements in the tangent space are of rank at most $2k$, a careful implementation of the tangent space projection (see~\cite{Vandereycken2013,bookch2020}) yields a decomposition
\[
X_\ell - \alpha_\ell \matop P_\ell( \matop{A}^*(\matop{A}(X_\ell) - y)) = \hat U K \hat V^\top 
\]
where $\hat U \in {\mathbb R}^{M \times 2k}$ and $\hat V \in {\mathbb R}^{N \times 2k}$ both have pairwise orthonormal columns. To apply $\matop H_s$ one hence needs to find the $s$ largest rows of $\hat U K$, which has complexity $O((s+k^2)M)$ (since $k \le s$), as opposed to $O((s+N)M)$ in Algorithm~\ref{alg: IHT quasi}. Since $\hat V$ is already orthogonal, the subsequent computation of a best rank-$k$ approximation requires only an SVD of the resulting $s \times 2k$ matrix (cost $O(k^2s)$), as opposed to an $s \times N$ matrix (cost $O(s^2 N)$ if $s \le N$). A comparison including the cost of forming $\matop{A}^*(\matop{A}(X_\ell) - y)$ is made in section~\ref{sec: complexity}.

\begin{remark}
Formally Algorithm~\ref{alg: semi-riem IHT} is well defined only as long as the iterates remain of full rank $k$. In the special case where $X_\ell$ has rank lower than $k$, we can slightly abuse the above notation and let $\mathbf P_\ell$ denote the projection on the tangent cone, which is given,~e.g., in~\cite{Schneider2016}. Since the tangent cone is symmetric in $0$, it indeed holds $- \matop P_\ell \nabla f(X) = \matop P_\ell (-\nabla f(X))$. In practice, the rank usually never drops and this issue can be ignored, except for the starting point in zero, which for this reason we have stated explicitly as $X_1= (\matop T_k \circ \matop H_s) (\alpha_0 \matop{A}^* y)$. Here, the initial step size is calculated using the Armijo-rule~\eqref{eq:riemarm} with $X_0 = 0$.
\end{remark}

\begin{remark}
\label{rem:order}
It would also make sense to use the quasi-projection $\matop H_s \circ \matop T_k$ instead of $\matop T_k \circ \matop H_s$, that is,
\[
X_{\ell+1} = (\matop H_s \circ \matop T_k)(X_\ell - \alpha_\ell \matop P_\ell(\nabla f(X_\ell)) ).
\]
This can be interpreted as Riemannian gradient method on $\mathcal M_k$ with retraction~$\matop T_k$, see~\cite{Absil2008,Absil2012}, but with additional thresholding by $\matop H_s$. Depending on $s$ and $k$, this order of projections can be implemented even more efficiently in many cases. On the other hand, our experiments have shown that the improvement is negligible unless $s \gg 2k$, in particular taking the more costly gradient computation into account. For the choice of initial point $X_1$, however, it appears to be {\em very} important to truncate the $M-s$ smallest rows of $\matop A^*(y)$ before the rank-$k$ truncation as in Algorithm~\ref{alg: semi-riem IHT}, and not the other way round, as this greatly improves the success rate. After that initialization, we did not observe a significant difference of the two orderings and therefore kept it consistent with Algorithm~\ref{alg: IHT quasi}.
\end{remark}

We now present a local convergence result for Algorithm~\ref{alg: semi-riem IHT} with constant stepsize $\alpha_\ell=1$ and under similar RIP conditions as for Algorithm~\ref{alg: IHT quasi}. Note that the statement is the same as in Proposition~\ref{prop: asymptotic rate iht}.

\begin{proposition}\label{prop:localConvRIHT}
Let $X^*$ be a solution of~\eqref{eq: problem}  with exactly $s$ nonzero rows and assume $\rank(X^*) = \min(k,s)$. Assume the spectral norm of $\matop I-\matop{A}^*\matop A$ on $T_{X^*}\mathcal{M}_{k,S^*}$ is $\delta < 1$.  There exists an $\epsilon>0$ such that if the  the sequence $(X_{\ell})$  generated by Algorithm~\ref{alg: semi-riem IHT} with stepsize $\alpha_\ell=1$ satisfies \(\|X_\ell-X^*\|_F\le \epsilon\) for some $\ell$, then $X_\ell$ converges to $X^*$ and 
$\limsup_{\ell\to\infty}\frac{\|X_{\ell+1}-X^*\|_F}{\|X_{\ell}-X^*\|_F}\leq \delta$.
\end{proposition}

\begin{proof}
The proof is similar to the one of Proposition~\ref{prop: asymptotic rate iht}. We first note that in a neighborhood of $X^*$ the projection $\matop H_s$ equals the projection $D_{S^*}$ onto the row support of $X^*$, which is a linear operator represented by a diagonal matrix. Next, we also linearize the projection $\matop T_k$ at the point $X^*$, and approximate the tangent space projection $\matop \matop P_{T_{X_\ell}\mathcal M_k}$ by $\matop P_{T_{X^*}\mathcal M_K}$ using $\|\matop P_{T_{X_\ell}\mathcal M_K}- \matop P_{T_{X^*}\mathcal M_K}\|\leq c\|X_\ell-X^*\|_F$ in spectral norm for some $c > 0$ (for instance $c = \frac{1}{\sigma_{k}(X^*)}$, see, e.g.,~\cite[Lemma~4.2]{Wei2016}). We get
\begin{align*}
X_{\ell+1}-X^*&= (\matop T_k\circ\matop H_s )(X_\ell - \matop P_{T_{X_\ell}\mathcal M_k}\matop A\matop A^*(X_\ell -X^*))-X^*\\
&= \matop P_{T_{X^*}\mathcal{M}_k}D_{S^*} (X_\ell -\matop P_{T_{X^*}\mathcal M_k}\matop A\matop A^*(X_\ell -X^*))-X^*+o(\|X_\ell-X^*\|_F)\\
&=\matop P_{T_{X^*}\mathcal{M}_k}D_{S^*} (\matop I-\matop P_{T_{X^*}\mathcal{M}_k} \matop A\matop A^*)(X_\ell -X^*)+o(\|X_\ell-X^*\|_F),
\end{align*}
where for the last equality we have used $X^* \in T_{X^*} \mathcal M_k$. Let now $X^*=U\Sigma V^\top$ be a singular value decomposition. Since $U$ has row support $S^*$, we have
\[
D_{S^*}(UU^\top Z+Z VV^\top-UU^\top ZVV^\top)=UU^\top D_{S^*} Z+ D_{S^*} ZVV^\top -UU^\top D_{S^*} Z VV^\top
\]
for any $Z$. Recalling the formula~\eqref{eq: tangent space projector}, this shows that the projections $D_{S^*}$ and $\matop P_{T_{X^*}\mathcal{M}_k}$ commute. Therefore $D_{S^*} \matop P_{T_{X^*}\mathcal{M}_k} = \matop P_{T_{X^*} \mathcal M_k} D_{S^*} =  \matop P_{T_{X^*}\mathcal M_{k,S^*}}$. As in the proof of Proposition~\ref{prop: asymptotic rate iht} we arrive at
\[
 X_{\ell+1}-X^*=\matop P_{T_{X^*}\mathcal{M}_{k,S^*}}(\matop I-\matop A^*\matop A )\matop P_{T_{X^*}\mathcal{M}_{k,S^*}}(X^* -X_\ell)+o(\|X_\ell-X^*\|_F),
\]
which for any $\varepsilon > 0$ can be bounded by $(\delta + \varepsilon)\|X_\ell-X^*\|_F$ for $X_\ell$ close enough to $X^*$. This implies the local convergence at the asserted asymptotic rate.
\end{proof}

\subsection{Improved numerical complexity for rank-one measurements}\label{sec: complexity}

In practice, Algorithms~\ref{alg: IHT quasi} and~\ref{alg: semi-riem IHT} often perform equally well. The main difference is that Algoritm~\ref{alg: semi-riem IHT} uses the projected gradient on the tangent space of (the smooth part of) $\mathcal M_{k}$. Thus a potential performance gain is tied to the question whether the low dimensionality of these tangent spaces can be exploited to achieve a  lower computational complexity. It turns out that this is the case in the important scenario of rank-one measurements, which occurs frequently in the literature.

The main bottleneck in both algorithms is forming $\matop A^*(\matop A(X)-y)$ or its projected version. Rank-one measurements take the form
\[
 \langle A_p, X \rangle_F = \langle a^{}_p b_p^\top, X \rangle_F = a_p^\top X b^{}_p, \quad p=1,\dots,m.
\]
In this case, forming $\matop{A}(X_\ell) - y$ for an $X\in\mathcal M_{k,s}$ that is already in the form $X^{}_\ell = U^{}_\ell \Sigma^{}_\ell V_\ell^\top$ with $\| U_\ell \|_0 \le s$ needs only $O(k(s+N)m)$ flops. In the Riemannian version, the application of the dual operator and projection to the tangent space can be combined in the following way:
\begin{equation}\label{eq: efficient implementation}
\matop P_\ell\matop A^* (z) = U_\ell^{}  \left( \sum_{p=1}^m z_p U_\ell^\top a_p^{}b_p^\top\right)+\left(\sum_{p=1}^m z_p a_p^{}b_p^\top V_\ell^{} \right) V_\ell^\top -  U_\ell^{}\left( \sum_{p=1}^m z_pU_\ell^\top a_p^{}b_p^\top V_\ell^{}\right) V_\ell^\top.
\end{equation}
The cost for this is $O(k(N+M)m)$ since $U_\ell$ and $V_\ell$ as well as the matrices in the sums are  $M\times k$ and $N\times k$ matrices.
Note that in a careful implementation, only the terms in the brackets need to be computed to represent the tangent vector. From this representation, it is possible to apply the projections $\matop H_s$ and $\matop T_k$ efficiently as mentioned above. 
 
\begin{table}[t]
\centering
\setlength{\tabcolsep}{10pt}
\renewcommand{\arraystretch}{1.5}
\begin{tabu}{ l | c }
\toprule

{\bf Operation} & {\bf Computational cost} \\ 

\midrule

Application of general $\matop A$ and $\matop A^*$ &  $O(mNM)$ \\
Application of rank-one $\matop A$ & $O(mk(N+s))$\\
Application of rank-one $\matop A^*$&  $O(mNM)$\\
Application of rank-one $\matop P_\ell \matop A^*$  &$O(mk(N+M))$\\
$\matop H_s$ of a full-rank matrix & $O((s+N)M)$\\
$\matop H_s$ of a rank-$k$ matrix & $O((s+k)M)$\\
SVD of a row-sparse matrix &$O(s^2N)$\\

\midrule

Overall cost of Alg.~\ref{alg: IHT quasi} and~\ref{alg: semi-riem IHT} with general $\matop A$ & $O(mNM)$\\
Overall cost of Alg.~\ref{alg: IHT quasi} with rank-one $\matop A$& $O(mNM)$\\
Overall cost of Alg.~\ref{alg: semi-riem IHT} with rank-one $\matop A$& $O(mk(N+M))$\\

\bottomrule
\end{tabu}
\caption{Complexities of operations in Algorithms~\ref{alg: IHT quasi} and~\ref{alg: semi-riem IHT}.}
\label{tab: complexity}
\end{table}

In the non-Riemannian version in Algorithm~\ref{alg: IHT quasi} the tangent space projection is not applied. For rank-one measurements, $\matop{A^*}(\matop A(X) - y)$ is a sum of $m$ rank-one matrices, but this does not help since usually $m \ge N$.
The cost  remains $O(mNM)$.

We conclude that in the case of rank-one measurements, if $k$ is much smaller than $N$, the Riemannian method should be computationally beneficial. This is confirmed by our numerical experiment in section~\ref{sec: blind deconvolution}. Table~\ref{tab: complexity} contains the complexities for the main steps in both algorithms. Note that unlike for Gaussian measurements, we usually cannot expect an RIP to hold for rank-one measurements and therefore even the local convergence result in Proposition~\ref{prop:localConvRIHT} might not be applicable. It would be interesting to study under which conditions the contractivity of $\matop I -\matop A^*\matop A$ on the tangent space $T_{X^*}\mathcal{M}_{k,S^*}$ as required in this proposition can be guaranteed for rank-one measurements, but we do not pursue this here.

\subsection{Soft-thresholding as a Riemannian proximal gradient method}\label{sec: Riemannian prox-gradient method}

In the following, we consider the case where the rank $k$ is known but the sparsity parameter $s$ is not. Our main application of blind deconvolution falls exactly into this category for the special case $k=1$. Both methods proposed above can be made adaptive with respect to $s$ by selecting in every step a threshold on the row norm to decide which rows to keep. A well established approach is soft thresholding. Here we show that such an approach can be interpreted as a Riemannian proximal gradient method on the manifold $\mathcal M_k$.  
We remark that soft thresholding could in principle also be applied to the rank if it is unknown but this case is not considered.

The method is derived as follows. For unknown $s$, to promote a row-sparse solution it is common to use the $(1,2)$-norm
\[
 g(X) = \| X \|_{1,2} = \sum_{i=1}^M \| X(i,:) \|_2
\]
as a convex penalty. Here, $X(i,:)$ denotes the $i$-th row of $X$. The task is then to minimize the function $f(x) + \mu g(x)$ with a penalty parameter $\mu > 0$. Note that $g$ is not differentiable in points $X$ having zero rows. Since both the function $f$ and $g$ are convex on $\R^{M \times N}$, an intuitive approach would be to consider methods like the proximal gradient descent as presented, e.g., in~\cite{Boyd2013,Beck2017}. These methods consist in applying the so called prox operator to $\mu g$ after the gradient step for $f$. However, prox operators are usually defined on convex domains. In our case, we consider a non-convex definition on the manifold $\mathcal M_k$ instead:
\begin{equation}\label{eq: prox on manifold}
 \prox_{\mu g}^{\mathcal M_k}(Y) \in \argmin_{X \in \mathcal M_k} \left( \mu g(X) + \frac{1}{2} \| X - Y \|_F^2 \right).
\end{equation}
For general $g$, we cannot evaluate such an operator easily, as it technically involves optimization of a local Lipschitz function on a manifold. However, as it turns out, for the particular choice of the $(1,2)$-norm, and for inputs $Y \in \mathcal M_k$, the prox operator simply coincides with the prox operator on the full space $\R^{M \times N}$, since the latter does not increase the rank. Its closed form solution is given via soft thresholding of rows. For completeness we provide a proof of this observation.

\begin{proposition}
\label{thm:softthresh}
For given $Y \in \R^{M \times N}$ and $\mu > 0$, the prox operator for the function $\mu g$ on $\R^{M \times N}$ has the closed form
\begin{equation}
\prox_{\mu g} (Y) \coloneqq \argmin_{X \in \R^{M \times N}} \left( \mu g(X) + \frac{1}{2} \| X - Y \|_F^2 \right) = \matop S_{1,2}^\mu(Y),
\end{equation}
where for each row $\mathbf y_i \in \R^N$ of $Y$, $\mathcal S_{1,2}^\mu(Y)$ is the soft thresholding operator
\begin{equation}\label{eq: ST operator}
\matop S_{1,2}^\mu(\mathbf y_i) \coloneqq 
\begin{cases}
\frac{\| \mathbf y_i \| - \mu}{\| \mathbf y_i \|} \; \mathbf y_i, \qquad & \text{if $\| \mathbf y_i \| > \mu$,} \\
0, \qquad &\text{otherwise.}
\end{cases}
\end{equation}
In particular, if $Y \in \mathcal M_k$, then also $\prox_{\mu g}^{\mathcal M_k}(Y) = \matop S_{1,2}^\mu(Y)$.
\end{proposition}
\begin{proof}
Since the function $g$ is convex in the ambient space, there exists exactly one solution 
\begin{equation*}
X^* = \prox_{\mu g} (Y).
\end{equation*}
The equivalent optimality condition is $0 \in \partial g(X^*) + \frac{1}{\mu}(X^* - Y)$. For each nonzero row $\mathbf x_\ell^* \neq 0$ of $X^*$ this means
\[
0 = \frac{\mathbf x_\ell^*}{\| \mathbf x_\ell^* \|} + \frac{1}{\mu} (\mathbf x_\ell^* - \mathbf y_i) = \left( 1 + \frac{\| \mathbf x_\ell^* \|}{\mu}\right) \frac{\mathbf x_\ell^*}{\| \mathbf x_\ell^* \|} -  \frac{1}{\mu}\mathbf y_i.
\]
Due to $\mu > 0$ this is only possible if $\| \mathbf y \| > \mu$, in which case we must have $\mathbf x_\ell^* = \frac{\| \mathbf y_i \| - \mu}{\| \mathbf y_i \|}$. This shows~\eqref{eq: ST operator}. For the second statement we first note that $\matop S_{1,2}(Y)$ acts on $Y$ by multiplication of a diagonal matrix, and therefore does not increase the rank. Since the argmin in~\eqref{eq: prox on manifold} is taken over a subset of $\R^{M \times N}$ we must have $\prox_{\mu g} (Y) = \prox^{\mathcal M_k}_{\mu g} (Y)$ if $Y \in \mathcal M_k$.
\end{proof} 

By analogy to proximal gradient methods we combine the prox operation on $\mathcal M_k$ with a Riemannian gradient descent for minimizing $f$ on $\mathcal M_k$. Using again the inherited Euclidean metric on $\mathcal M_k$ and $\matop T_k$ as a retraction, this results in the following iteration
\begin{align*}
X_{\ell+1} = \prox_{\mu g}^{\mathcal M_k} \Bigl( \mathbf T_k \bigl(X_\ell - \alpha_\ell \matop P_\ell(\nabla f(X_\ell))\bigr) \Bigr) = (\matop S_{1,2}^\mu \circ \matop T_k)(X_\ell - \alpha_\ell \matop{P}_\ell \matop{A}^*(\matop{A}(X_\ell) - y))),
\end{align*}
which can be regarded as a Riemannian version of proximal gradient descent. Note that this formulation differs from other possible generalizations of proximal gradient methods on manifolds~\cite{Chen2020,Huang2020} which are based on minimization of quadratic models on the tangent spaces for finding  an appropriate search direction. In our formulation above, while only applicable in this specific setup, the closed form solution of the prox operator on the manifold is available, which makes it a very intuitive alteration of the original algorithm. The full scheme is displayed in Algorithm~\ref{alg:RiemProxGrad}.

\begin{algorithm}
\SetKwInOut{Input}{Input}\SetKwInOut{Output}{output}
\Input{Linear operator $\matop{A}$, measurements $y$, $s_0 \in \mathbb N$, $\mu \in \R^+$, $\tau \in (0,1)$,\\ starting point $X_1= (\matop T_k \circ \matop H_{s_0}) (\alpha_0 \matop{A}^* y) \in \mathcal{M}_{k}$ with initial step size $\alpha_0 \in \R$}
\For{$\ell = 1,2,\dots$}
{
Choose stepsize $\alpha_\ell > 0$\;
Compute $\mu = \tau^\ell \max_{k}(\| \matop T_k(X_\ell - \alpha_\ell \matop{P}_\ell \matop{A}^*(\matop{A}(X_\ell) - y))_i \| : i = 1,\ldots,M )$\;
Compute
\[
X_{\ell + 1} = (\matop S_{1,2}^\mu \circ \matop T_k)(X_\ell - \alpha_\ell \matop{P}_\ell \matop{A}^*(\matop{A}(X_\ell) - y));
\]
}
\caption{A Riemannian proximal gradient method.}\label{alg:RiemProxGrad}
\end{algorithm}

In the proposed algorithm the thresholding parameter $\mu$ is reduced by a factor $\tau$ in each iteration. Different values of $\tau$ can be used depending on the problem. Other heuristics for selecting $\mu$ are possible as well. We comment on our implementation of the algorithm in the experiment section.

Again, we suggest to initialize the algorithm with $(\matop T_k \circ \matop H_{s_0})(\alpha_0 \matop A^* y)$, where $s_0$ is a guess for the a priori unknown sparsity $s$ of the solution and $\alpha_0$ is the Armijo step size. As we have already emphasized above, the choice of the starting point has proven to be an important step and the success of the Riemannian proximal gradient method will be somewhat limited by the missing knowledge of $s$. In the experiments, we picked $s_0 = \min(M,(m + k(k - N))/k)$, which is the maximal row sparsity that can in theory be detected with a given number of measurements $m$ from the degrees of freedom, but there was no significant improvement compared to a non-sparse starting point.

For the soft thresholding parameter $\mu$, we propose the rule
$\mu =\tau^\ell \max_{k}(\| \matop T_k(X_\ell - \alpha_\ell \matop{P}_\ell \matop{A}^*(\matop{A}(X_\ell) - y))_i \| : i = 1,\ldots,M )$, i.e., the norm of the $k$-th largest row of the current iterate. For sufficiently large $\tau$, this will set all but $k$ of the rows of $X_2$ to zero but ensures that the one with the largest norms remain active.

For determining the step sizes $\alpha_\ell$, we again propose to use a line search method based on gradient step only, that is,
\[
f(X_\ell) - f\bigl(\matop T_k(X_\ell - \beta^p \matop P_\ell( \matop{A}^*(\matop{A}(X_\ell) - y)))\bigr) \geq \gamma \, \beta^p \, \| \matop P_\ell( \matop{A}^*(\matop{A}(X_\ell) - y))\|_F^2.
\]

\begin{remark}
We remark that Algorithm~\ref{alg:RiemProxGrad} can be formally derived from the Riemannian IHT method in Algorithm~\ref{alg: semi-riem IHT} by switching to the quasi-optimal projection ${\hat{\matop P}}_{k,s} = \matop H_s  \circ \matop T_k$ (cf.~Remark~\ref{rem:order}), and then replacing the hard thresholding operator $\matop H_s$ with the operator $\matop S^\mu_{1,2}$. For soft thresholding this order of first truncating the rank before selecting the rows indeed is convenient due to Proposition~\ref{thm:softthresh}. The potential alternative of applying first soft thresholding and then rank truncation caused inconsistent behavior in the Armijo line search in our experiments. 
\end{remark}

\section{Numerical experiments}\label{sec: numerical experiments}

In this section we present some results of numerical experiments with the proposed algorithms. In the first set of results, we consider recovery of synthetic data using Gaussian measurements. In the second, we use random rank one measurements and also test the algorithms for a blind deconvolution problem.

\subsection{Recovery with Gaussian measurements}\label{sec: numerical experiments generic}

For the recovery problem~\eqref{eq: problem} we compare the success rates of Algorithms~\ref{alg: IHT quasi} and \ref{alg: semi-riem IHT} for different row sparsity levels $s$ and different column sizes $N$ when using random Gaussian measurements. Specifically, we generate a random matrix $X^* \in \mathcal M_{k,s}$ and take $m$ measurements $y_p = \langle A_p, X^* \rangle_F$ with normally distributed $A_p \sim\mathcal N(0,\frac{1}{\sqrt{m}})$. We fix the row dimension $M = 1000$ and the rank $k = 3$.

Figure~\ref{fig:loglog} shows a phase transition plot for different numbers of measurements ($m = \mathrm{round}(1.2^j), j = 18,\ldots,36$) on the y-axis and different row-sparsity ($s = \mathrm{round}(1.2^j), j = 6,\ldots,15$) on the x-axis. The values for $m$ and $s$ were empirically chosen because they yielded the most expressive results. The grayscale denotes the success rate for parameter setting, where white means no success and black means $100\%$ success. Both algorithms were tested with a fixed stepsize $\alpha_\ell = 1$ (on the left) and with an adaptive stepsize using an Armijo linesearch (on the right). The column size is always taken to be equal to the sparsity, that is, $N = s$, and we performed each experiment $10$ times. The initial points were taken as $X_1 = (\matop T_k \circ \matop H_s)(\matop A^* y)$, which we found to be crucial for the overall performance.

\begin{figure}[t]
\includegraphics[width=\textwidth]{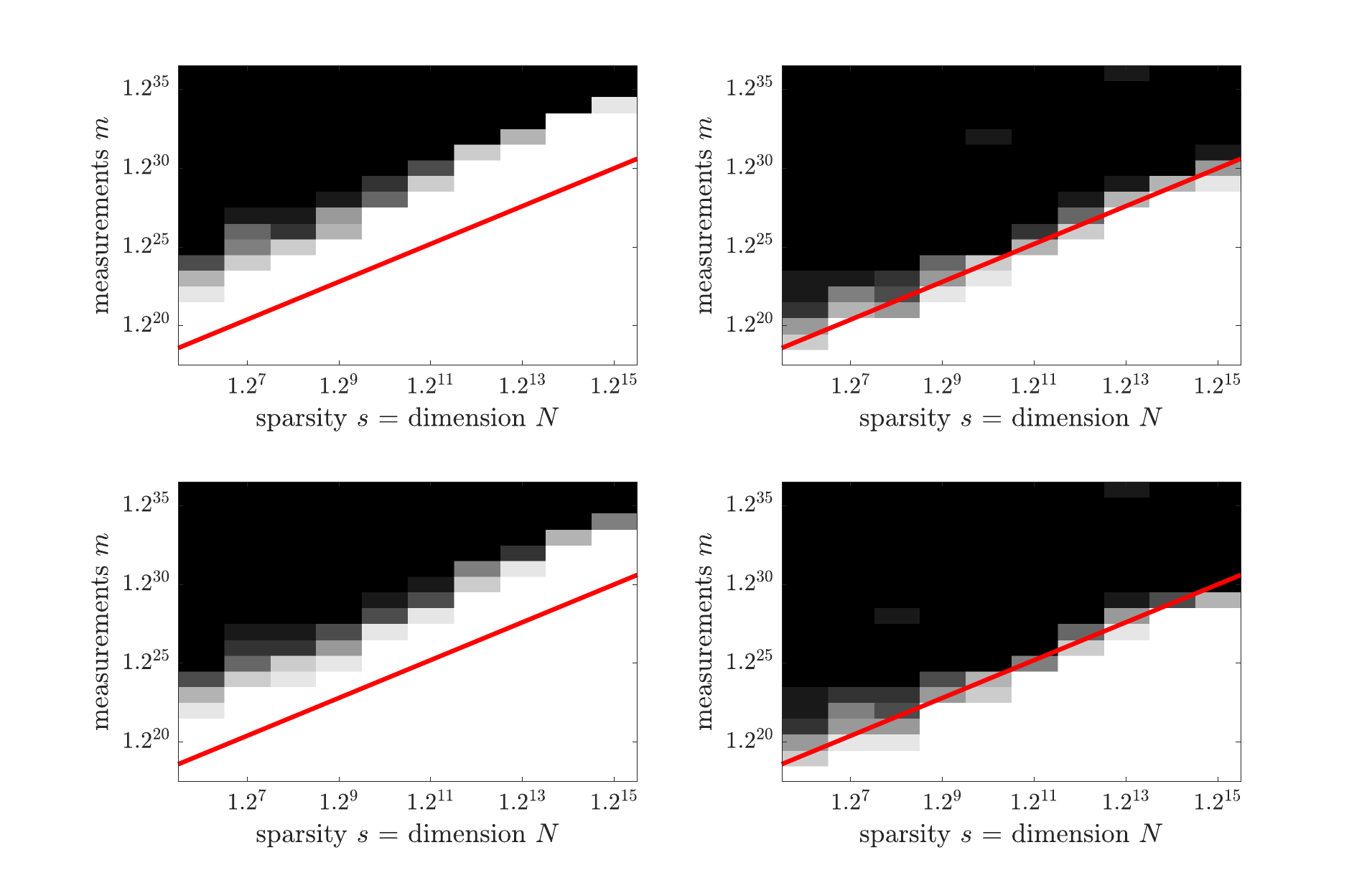}
\caption{Success rate of IHT (upper left), adaptive IHT (upper right), Riemannian IHT (lower left) and Riemannian adaptive IHT (lower right) out of 10 tries on a loglog scale. The row-sparsity $s$ is on the $x$-axis and equals the column size $N$, the number of measurements $m$ is on the $y$-axis. The rank is $k = 3$ and row size $M = 1000$. The red line has slope one, indicating linear dependence on $s$ and $N$.}
\label{fig:loglog}
\end{figure}

We can see that the algorithms with adaptive stepsize are in general more often successful. Note that the plots are provided on a loglog scale. Therefore, since we have set $N= s$, a line with slope one (depicted in red) indicates linear dependence on $s+N$ (as opposed to, e.g., linear dependence on $sN$, which would have slope two). The right plots therefore indeed suggest such a linear dependence $m = O(s+N)$ for successful recovery when adaptive stepsizes are used. Recall that for a fixed rank $k$ such a scaling is optimal. In the left plots with fixed stepsize it is a bit more difficult to recognize the slope of the transition line, which could be slightly larger than one. 

\begin{figure}[t]
\includegraphics[width=\textwidth]{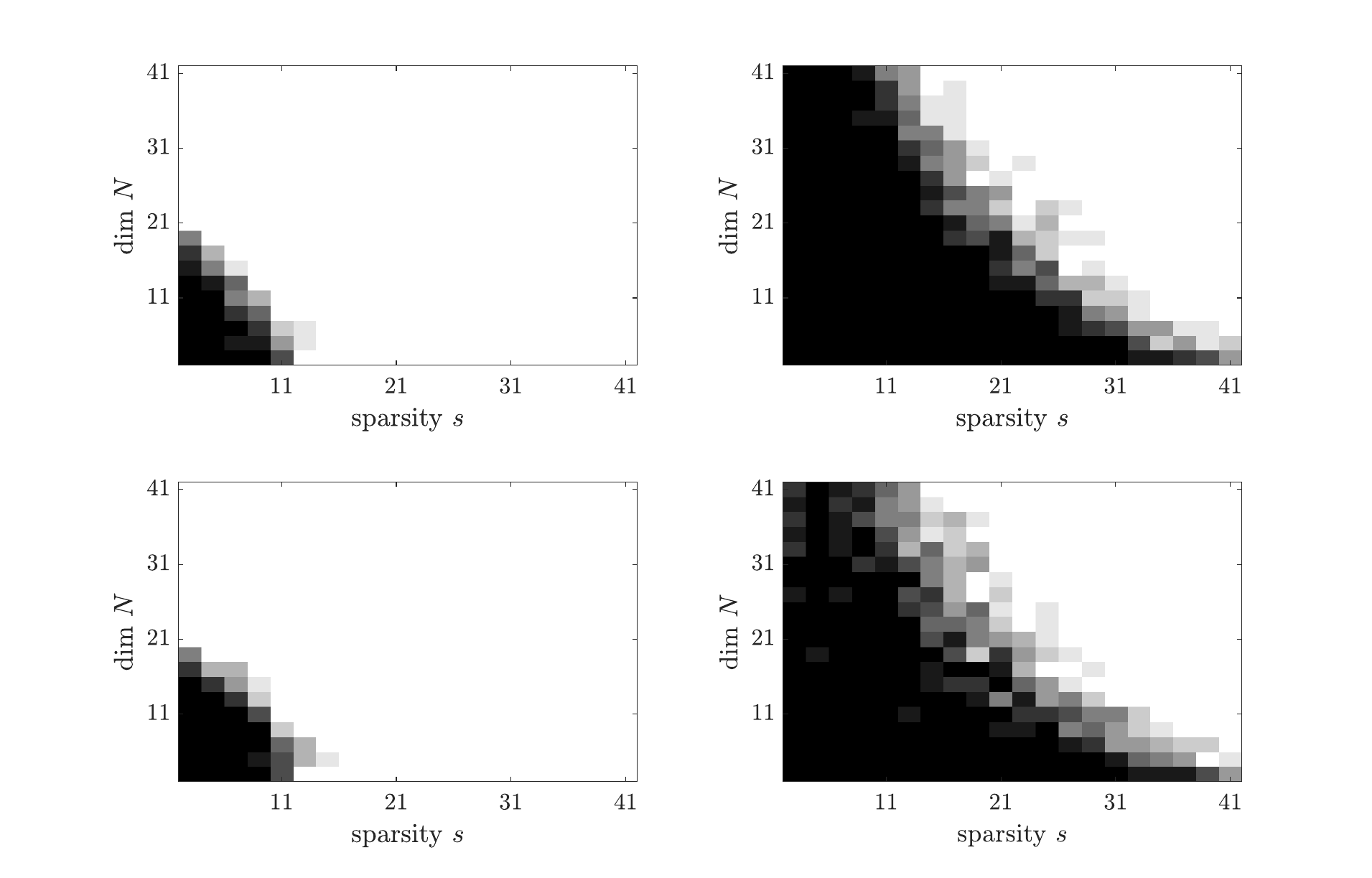}
\caption{Success rate of IHT (upper left), adaptive IHT (upper right), Riemannian IHT (lower left) and Riemannian adaptive IHT (lower right) out of 10 tries. The row-sparsity $s$ is on the $x$-axis, the column size $N$ is on the $y$-axis. The number of measurements is $m = 300$, the rank is $k = 3$ and the row size is $M = 1000$.}
\label{fig:sN}
\end{figure}

In a second experiment we fixed the number of Gaussian measurements $m = 300$ and varied the row-sparsity $s$ and the column size $N$ independently. The row size was again $M = 1000$ and rank $k = 3$. The results are given in Figure~\ref{fig:sN}. Note that the axes have a linear scale in this experiment. The algorithms with adaptive stepsize perform clearly better. The precise relation between $s$ and $N$ for the transition curve is, however, difficult to assess from these plots.

\begin{figure}[t]
\includegraphics[width=\textwidth]{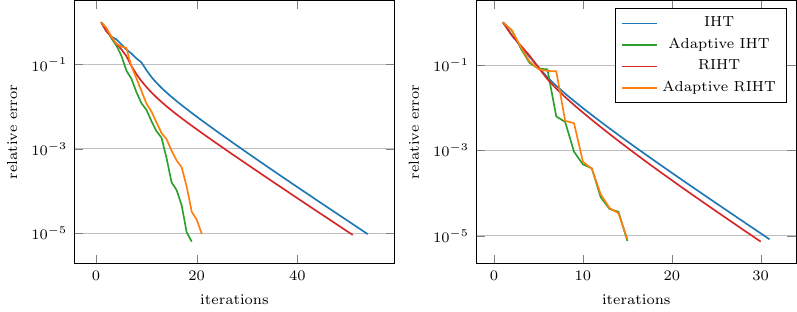}
\caption{Number of iterations against relative error for the three methods. The number of measurements is $m = 520$ (left) and $m = 800$ (right), the sparsity $s = 20$ and the dimension $N = 10$. The rank is $k = 3$ and $M = 1000$.}
\label{fig:convergence}
\end{figure}

We can also compare the convergence speed of each algorithm in terms of iteration numbers. Figure~\ref{fig:convergence} shows the relative errors to the exact solution for two different parameter settings, one that was borderline in the previous experiments (on the left) and another one for which all algorithms find the solution with ease (on the right). The observed behaviour, however, was actually almost the same for other cases. We can see that the methods with adaptive stepsize converge faster, perhaps even superlinearly, although they are of course more costly. For fixed stepsize, the Riemannian method outperforms its classical counterpart but the rate of convergence is the same.

\begin{figure}[t]
\centering
\includegraphics[width=.495\textwidth,trim=0 -35 0 0]{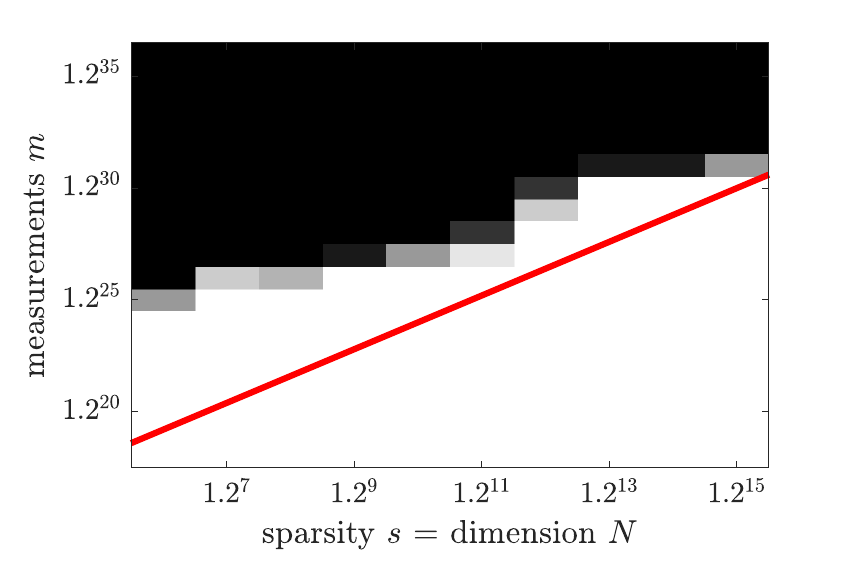}
\includegraphics[width=.495\textwidth]{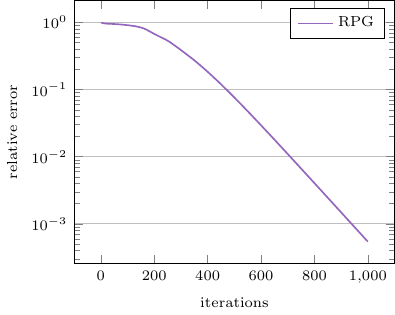}
\caption{Left: Success rate of the Riemannian proximal gradient method out of 10 tries on a loglog scale. The sparsity $s$ is equal to the matrix dimension $N$ on the $x$-axis, the number of measurements $m$ is on the $y$-axis. The rank is $k = 3$ and $M = 1000$. The red line has again slope $1$ and corresponds to linear dependence on $s$ and $N$.
Right: Number of iterations against relative error. The number of measurements is $m = 520$, the sparsity $s = 20$ and the dimension $N = 10$. The rank is $k = 3$ and $M = 1000$.}
\label{fig:loglog_prox}
\end{figure}

Finally, we present a proof of concept for the Riemannian proximal gradient method in Algorithm~\ref{alg:RiemProxGrad},~see Figure~\ref{fig:loglog_prox}. Specifically, we tested this method in the same setting as the first experiment. The factor $\tau$ that decreases the thresholding parameter $\mu$ in each step was set to $\tau = 0.99$, that is, $\mu$ is decreased by 1\% per step. We used adaptive stepsizes with linesearch and the initial guess $X_1 = (\matop T_k \circ \matop H_{s_0})(\alpha_0 \matop A^* y)$, as discussed in section~\ref{sec: Riemannian prox-gradient method}. As can be seen, the success rate of this method is lower than for the previous algorithm, which is natural since the sparsity parameter $s$ is unknown here. The deviation from the red line with slope 1 could be due to the effect of the different initialization, which is more prominent for small sparsity. Yet, for larger $s$, the dependence of the required measurements $m$ on $s$ and $N$ seems linear and hence optimal as well. We also repeated the borderline case from Figure~\ref{fig:convergence} and observe slow but linear convergence.

\subsection{Rank-one measurements and blind deconvolution}\label{sec: blind deconvolution}

We now examine the case of rank-one measurements, where
\[
 \langle A_p, X \rangle_F = \langle a_p b_p^\top, X \rangle_F = a_p^\top X b_p, \quad p=1,\dots,m.
\]
As discussed in section~\ref{sec: complexity}, using rank-one measurements enables a more efficient evaluation of gradients and tangent space projections. In particular, when implemented accordingly we expect the Riemannian version of IHT to be faster than the standard version.

We consider two experiments with rank-one measurements. In the first we take random rank-one measurements on synthetic data. Specifically we choose $a_p \sim\mathcal N(0,1)$ and $b_p \sim\mathcal N(0,\frac{1}{\sqrt{m}})$ to closely match the setting of random gaussian measurements (that is, with the correct scaling).

In the second experiment we use deterministic rank-one measurements based on the discrete Fourier transform. This setting can be motivated from applications in blind deconvolution. Consider the convolution 
\[
w * z=\left(\sum_{\ell=1}^m w_\ell z_{j-\ell}\right)_{j=1}^m,
\] 
of two real vectors of length $m$ where the indices are to be considered modulo $m$. The inverse operation, where both $w$ and $z$ are reconstructed from their convolution $w*z$, is called \emph{blind deconvolution}. In general, this is of course an ill-posed problem. A common assumption that renders a recovery possible is that $w$ and $z$ lie in some known subspaces, that is, $w= Bu$ and $z=Cv$ for some $B \in \R^{m \times M}$ and $C \in \R^{m \times N}$. As suggested in~\cite{Ahmed2014}, one can then recast the problem as a linear recovery task for a rank-one matrix. More precisely, one can diagonalize the action of $*$ using the (unitary) discrete Fourier transform~$F = [\frac{1}{\sqrt{m}}\exp(- \frac{i 2 \pi (k-1)(j-1)}{ m})]_{jk}$ which yields
\[
y = F(w*z)=\sqrt{m} \diag(Fw)F(z)=\sqrt{m} \diag(FBu)FCv=\matop A(uv^\top).
\]
Here, the last equality implicitly defines the linear operator $\matop A\colon \R^{M\times N}\to \R^m$. This is possible since every bilinear map in $(u,v)$ can be lifted to a linear map acting on $uv^\top$. In certain applications, the vector $u$ can also be assumed to be sparse. We therefore obtain an instance of our problem \eqref{eq: problem} with $k=1$. For further references, see e.g.~\cite{Jung2018}.

To see that the operator $\matop A$ defined in this way performs rank-one measurements one verifies that
\[
\langle A_p , uv^\top \rangle_F = \sqrt{m} \langle (FB)^{\mathsf H}_{p,:} \overline{(FC)^{}_{p,:}} , uv^\top \rangle_F,
\]
$(FB)_{p,:}$ and $(FC)_{p:}$ denote the $p$-th row of $FB$ and $FC$ respectively. Indeed, after a suitable reshape, the operator $\matop A$ effectively becomes the (row-wise) Khatri-Rao product of $FB$ and $FC$, that is,
\begin{equation*}
A_{p,[i,j]} = \sqrt{m} (FB)_{p,i} (FC)_{p,j}
\end{equation*}
This representation allows us to show that $\matop A^* \matop A$ is a real operator, since
\begin{align*}
(A^{\mathsf H} A)_{[i_1,j_1],[i_2,j_2]} &= \sum_{p = 1}^m \overline{A_{p,[i_1,j_1]}} A_{p,[i_2,j_2]} \\
&= m \sum_{k_1=1}^m \sum_{\ell_1=1}^m \sum_{k_2=1}^m \sum_{\ell_1=1}^m \overline{B_{k_1,i_1}} \overline{C_{\ell_1,j_1}} B_{k_2,i_2} C_{\ell_2,j_2}  \cdot  \sum_{p=1}^m \overline{F_{p,k_1}} \overline{F_{p,\ell_1}} F_{p,k_2} F_{p,\ell_2}. 
\end{align*}
Since $B$ and $C$ are real matrices, it suffices to show that the last sum is real. But this holds since for $p,k,\ell \in \{1,\ldots,m\}$, by elementary manipulations,
\begin{align}
F_{p,k} F_{p,\ell} = \frac{1}{\sqrt m} F_{p,\left((k+\ell-2)\!\!\!\!\!\mod m \right) + 1}
\end{align}
and the rows and columns of $F$ are unitary. Therefore, while $y = \matop A(uv^\top)$ is a complex vector, the problem itself as well as all steps in the algorithm remain real. For the efficient implementation of the action of $\matop A$ and $\matop A^*$ as in~\eqref{eq: efficient implementation}, however, some obvious modifications are required.

\begin{figure}[t]
\centering
\includegraphics[width=\textwidth]{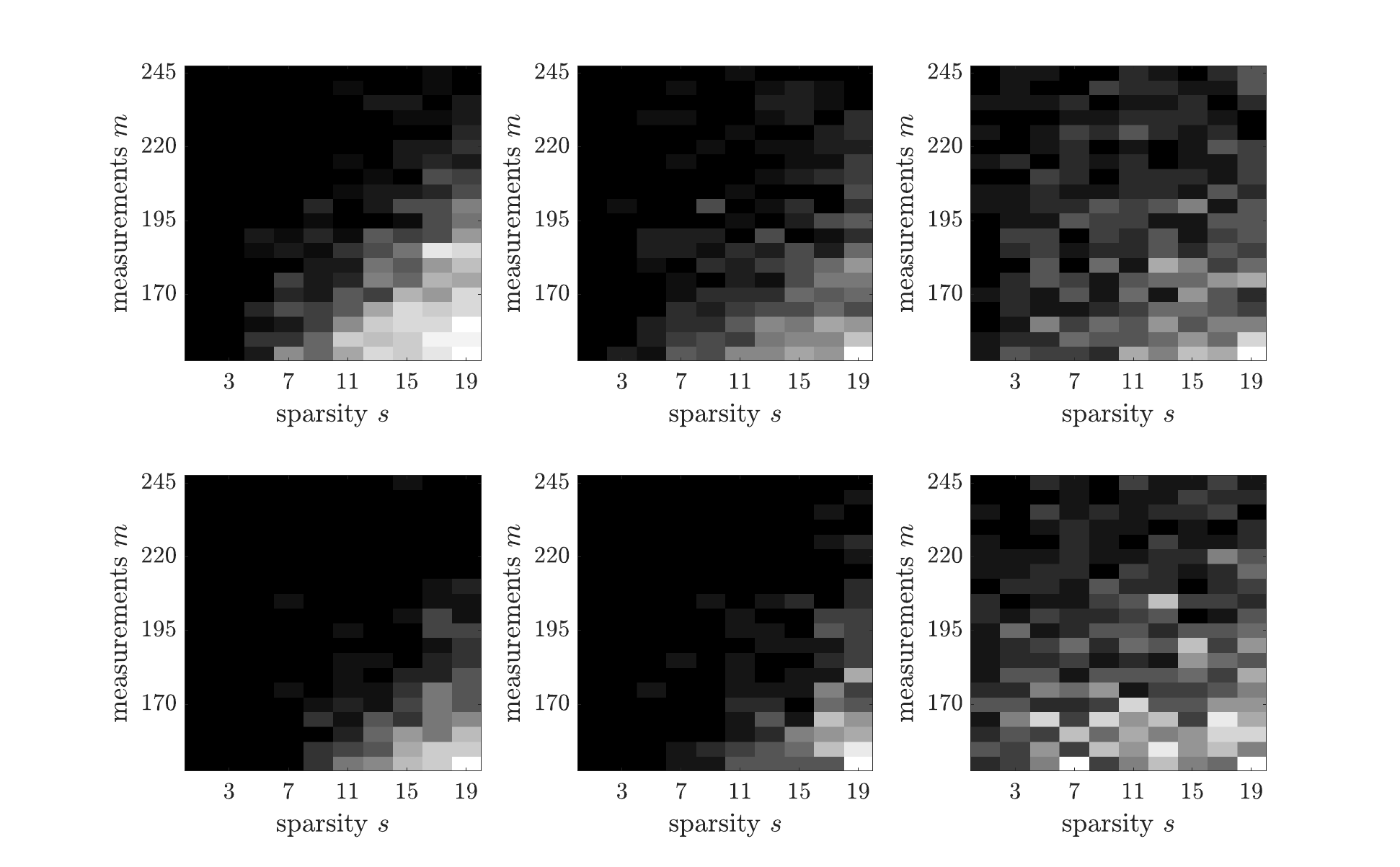}
\caption{Success rate of adaptive IHT (left), Riemannian adaptive IHT (middle) and Riemannian proximal gradient method (right) out of 20 tries using for random rank-one measurements (top) and the discrete Fourier transform (bottom). The matrix dimensions are $M = 150$ and $N = 50$. The sparsity $s$ is on the $x$-axis, the number of measurements $m$ is on the $y$-axis. The rank is $k = 1$.}
\label{fig:pt-robd}
\end{figure}

In both experiments we set $M = 150$ and $N = 50$. An exact solution $X^* = u v^\top$ of rank $k=1$ is generated by picking a random matrix $\hat X \sim\mathcal N(0,1)$ of size $s \times N$, computing its best rank-one approximation, and randomly distributing the resulting rows in a matrix of size $M \times N$. We then run the three algorithm with input $y = \matop A(X^*)$ for different values $m$ of rank-one measurements. For the Riemannian proximal gradient method, we chose the decrease of the thresholding to be $\tau = 0.999$, that is, a 0.1\% decrease in each step. 

In Figure~\ref{fig:pt-robd}, we show the phase transition plot for the two settings and the algorithms with adaptive stepsize, which performed better in the general setting and for the Riemannian proximal gradient method with unknown sparsity $s$. Again, the grayscale denotes the success rate for the different parameters $m$ and $s$. We performed $20$ tries for each setting as this yielded a sharper outline of the success rate. In Table~\ref{tab:convergence_bd} we report computational times and iteration numbers for both settings and the three algorithms in terms of the relative error. Here, the number of measurements was $m = 200$ and the row sparsity $s = 3$. Note that for $m \sim M + N = 200$, one expects convergence even without the sparsity constraint, however, we have observed that this is true only up to a large constant. We implemented the three methods in a comparable fashion, exploiting the structure of the rank one measurements as discussed in section~\ref{sec: complexity}. The computing time was measured on an Intel Core i7-10510U with 16 GB memory.

We can see that the adaptive IHT and the adaptive Riemannian IHT perform well in these experiments, especially for the Fourier measurements. The Riemannian method can be slightly better in terms of recovery, and significantly faster than the adaptive standard IHT method ($\sim 40\%$ improvement for random measurements and $50\% - 60\%$ for Fourier measurements).

\begin{table}[t]
\centering
\setlength{\tabcolsep}{10pt}
\renewcommand{\arraystretch}{1.5}
\begin{tabu}{ c | c  c | c  c | c  c }
\toprule
 & \multicolumn{2}{c|}{\bf Adaptive IHT} & \multicolumn{2}{c|}{\bf Adaptive RIHT} & \multicolumn{2}{c}{\bf RPG} \\
$\varepsilon$ & Iterations & CPU time & Iterations & CPU time & Iterations & CPU time \\ 

\midrule
\multicolumn{7}{l}{\bf Random Rank One Measurements} \\

$10^{-1}$ & 196 & 0.4131s & 196 & 0.2586s & 6691 & 9.5539s \\
$10^{-3}$ & 831 & 1.7495s & 831 & 1.1103s & 11490 & 16.675s \\
$10^{-5}$ & 1583 & 3.3613s & 1582 & 2.0928s & 16095 & 23.931s \\

\multicolumn{7}{l}{\bf Fourier Measurements} \\

$10^{-1}$ & 8 & 0.1455s & 10 & 0.0691s & 4211 & 25.331s \\
$10^{-3}$ & 25 & 0.4592s & 30 & 0.2189s & 8941 & 58.069s \\
$10^{-5}$ & 45 & 0.8268s & 46 & 0.3392s & 13545 & 91.412s \\
\bottomrule
\end{tabu}
\caption{Relative error of adaptive IHT, Riemannian Adaptive IHT and Riemannian proximal gradient method against number of iterations and CPU time for the setting of random rank-one measurements and discrete Fourier measurements. The number of measurements is $m = 200$, the sparsity is $s = 3$ and the dimensions are $M = 150$, $N = 50$. The rank is $k = 1$.}
\label{tab:convergence_bd}
\end{table}

The Riemannian proximal gradient method is capable of detecting the row sparsity but it has a lower success rate, and is also quite slow. We have found that this is almost entirely due to the choice of the starting point that can be chosen without the knowledge of the sparsity parameter $s$. Therefore, this algorithm can clearly be improved upon with some extra work on the start point. In any case, the relatively good success rate makes this a promising approach for further research in cases where the sparsity is not known a priori.

\begin{small}

\bibliographystyle{plain}
\bibliography{references}

\end{small}

\end{document}